\documentclass{amsart}

\usepackage{hyperref}
\usepackage{array}
\usepackage{algorithm}
\usepackage[noend]{algpseudocode}
\usepackage{enumerate}
\usepackage{amsmath}
\usepackage{amssymb}
\usepackage{amsfonts}
 \newtheorem{thm}{Theorem}[section]

 \newtheorem{prop}[thm]{Proposition}
 \theoremstyle{definition}
 \newtheorem{defn}[thm]{Definition}
 \theoremstyle{remark}
 \newtheorem{rem}[thm]{Remark}
 \newtheorem*{ex}{Example}
 \newtheorem{alg}{Algorithm}[section]
 \numberwithin{equation}{section}

\DeclareMathOperator{\ord}{ord} \DeclareMathOperator{\Card}{Card}

\DeclareMathOperator{\Ker}{Ker}

\DeclareMathOperator{\lcm}{lcm} \DeclareMathOperator{\lc}{lc}
\DeclareMathOperator{\Hom}{Hom}
\begin{document}
\def\D{\displaystyle}
%
%
%
\title[Multivariate Bernstein-type Polynomials of Finitely Generated $D$-Modules]{Multivariate Bernstein-type Polynomials of Finitely Generated $D$-Modules}

\author{Alexander Levin}
\address{Department of Mathematics, The Catholic University of
America, Washington, D.C. 20064}


\begin{abstract}
We generalize the Gr\"obner basis method for free $D$-modules to the case of several term orderings induced by a partition of the set of basic variables. Using this generalized Gr\"obner basis technique we
prove the existence and give a method of computation of a dimension polynomial in several variables associated with a finitely generated $D$-module. We show that this dimension polynomial carries more invariants of a $D$-module than the classical (univariate) Bernstein polynomial. We also show that the introduced multivariate dimension polynomial can be used for characterization of holonomic $D$-modules.

\end{abstract}

\maketitle
\section{Introduction}

In his fundamental work of 1971 \cite{B1} I. Bernstein introduced a Hilbert-type dimension polynomials for finitely generated modules over a Weyl algebras (also called $D$-modules) and extended the theory of multiplicity to the class of such modules. The results of this study have found many applications in singularity and monodromy theories; many analytical applications of Bernstein polynomials (also called Bernstein-Sato polynomials) can be found in a classical Bj\"ork's book \cite{Bj} and in a more recent monograph \cite{Co}. In particular, they allowed I. Bernstein \cite{B2} to prove Gelfand's conjecture on meromorphic extensions of functions $\Gamma_{f} (\lambda) = \int P^{\lambda}(x)f(x)dx$ in one complex variable $\lambda$ defined on the half-space $Re(\lambda) > 0$ for polynomials in $n$ real variables $P(x) = P(x_{1},\dots, x_{n})$  and for functions $f(x) = f(x_{1},\dots, x_{n})\in C_{0}^{\infty}(\mathbb{R}^{n})$.
The development of algorithmic methods for the study of $D$-modules started with the works of J. Brian, Ph. Maisonobe \cite{BM} and F.J. Castro-Jim\'{e}nez \cite{CJ} who adapted the theory of the Gr\"obner bases to Weyl algebras and free modules over them. These methods were essentially extended in works of M. Noro, T. Oaku, N. Takayama, N. (see \cite{N}, \cite{O1}, and \cite{O2}) and some other researchers. A good list of references on the subject can be found in a fundamental monograph \cite{SST} devoted to the computational analysis of holonomic systems and hypergeometric functions.

In this paper we extend the Gr\"obner basis method for free $D$-modules to the case of several term orderings that correspond to a fixed partition of the set of basic variables and the corresponding partition of the set of partial derivatives. We prove the existence, determine invariants, and outline methods of computation of multivariate Bernstein-type dimension polynomials associated with the natural multi-filtration that is determined by a set of generators of a finitely generated $D$-module and the above-mentioned partition of the set of variables. We show that such polynomials not only characterize the Bernstein class of $D$-modules, but also carry, in general, more invariants than the classical Bernstein polynomial. Note that the multivariate dimension polynomials introduced in this paper are of different type than bivariate dimension polynomials of filtered $D$-modules introduced in \cite{Levin2} and \cite{DL}. The latter polynomials are associated with the partition $X\bigcup\Delta$ of the set of generators of a Weyl algebra $A_{n}$, where $X=\{x_{1},\dots, x_{n}\}$ is the set of basic variables and $\Delta = \{\partial_{1},\dots, \partial_{n}\}$  is the set of the corresponding partial derivatives. These bivariate dimension polynomials carry more invariants of a finitely generated left $A_{n}$-module $M$ than the classical Bernstein polynomials, but not as many as the dimension polynomials introduced in our paper. (By an invariant of $M$ we mean a characteristic of this $A_{n}$-module that does not depend on the choice of a finite system of generators of $M$ over $A_{n}$.) Furthermore, the dimension polynomials introduced in this paper characterise holonomic modules (see Theorem 4.2 below) while the dimension polynomials introduced in \cite{Levin2} and \cite{DL} do not allow one to determine holonomicity.

\section{Preliminaries}

Throughout the paper, $\mathbb{N}$, $\mathbb{Z}$, and $\mathbb{Q}$ denote the sets of all non-negative integers, integers, and rational numbers, respectively; $\mathbb{Q}[t_{1},\dots, t_{p}]$ denotes the ring of polynomials in variables $t_{1},\dots, t_{p}$ over $\mathbb{Q}$. If $n\in\mathbb{Z}$, $n\geq 1$, then $\leq_{P}$ will denote the product order on $\mathbb{Z}^{n}$ (or $\mathbb{N}^{n}$), that is, a partial order $\leq_{P}$ such that $(a_{1},\dots, a_{n})\leq_{P}(a'_{1},\dots, a'_{n})$ if and only if $a_{i}\leq a'_{i}$ for $i=1,\dots, n$.

By a ring we always mean an associative ring with unity. Every ring homomorphism is unitary (maps unity to unity), every subring of a ring contains the unity of the ring. Unless otherwise indicated, by the module over a ring $R$ we mean a unitary left $R$-module. Furthermore, every field considered in this paper is supposed to have zero characteristic.

In what follows we consider a Weyl algebra $A_{n}(K)$ as an algebra of differential operators generated by the partial differentiations $\partial_{i} = \partial/\partial x_{i}$ (($1\leq i\leq n$) over a polynomial ring $R = K[x_{1},\dots, x_{n}]$ in $n$ variables $x_{1},\dots, x_{n}$ over a field $K$. Thus, $A_{n}(K)$ is a $K$-algebra generated by the elements $x_{1},\dots, x_{n}, \partial_{1}, \dots, \partial_{n}$
with defining relations $x_{i}x_{j} = x_{j}x_{i}$, $\partial_{i}\partial_{j} = \partial_{i}\partial_{j}$ ($1\leq i\leq j$),  $\partial_{i}x_{j} = x_{j}\partial_{i}$ if $i\neq j$, and $\partial_{i}x_{i} = x_{i}\partial_{i} + 1$ for $i=1,\dots, n$.  A left module over a Weyl algebra is said to be a {\em $D$-module}.

We shall denote multi-indices with non-negative integers by small Greek letters. Thus, monomials $x_{1}^{\alpha_{1}}\dots x_{n}^{\alpha_{n}}$ and $\partial_{1}^{\beta_{1}}\dots \partial_{n}^{\beta_{n}}$ are written as $x^{\alpha}$ and $\partial^{\beta}$, their total degrees $\alpha_{1}+\dots +\alpha_{n}$ and $\beta_{1}+\dots +\beta_{n}$ are denoted by $|\alpha|$ and $|\beta|$, respectively.

It is known (see \cite[Chapter 1, Proposition 1.2]{Bj}) that monomials $x^{\alpha}\partial^{\beta}$ ($\alpha, \beta \in\mathbb{N}^{n}$) form a basis
of $A_{n}(K)$ over the field $K$, so that every element $D\in A_{n}(K)$ can be written in a unique way as a finite sum $\sum_{\alpha, \beta} a_{\alpha \beta}x^{\alpha}\partial^{\beta}$ with the coefficients $a_{\alpha \beta}\in K$. If $D\neq 0$, then the number $\ord  D = \max\{|\alpha|+|\beta|\,|\,a_{\alpha \beta}\neq 0\}$ is called the order of the element $D$. We also set $\ord 0 = -\infty$.

Since $\ord (D_{1}D_{2}) = \ord D_{1} + \ord D_{2}$ for any $D_{1}, D_{2}\in A_{n}(K)\setminus\{0\}$, the Weyl algebra $A_{n}(K)$ can be considered as a filtered ring with the ascending filtration
$(W_{r})_{r\in\mathbb{Z}}$ where $W_{r} = \{D\in A_{n}(K) | \ord D\leq r\}$ for $r\in\mathbb{N}$ and $W_{r} = 0$, if $r<0$.

If $M$ is a finitely generated left $A_{n}(K)$-module with a system of generators $g_{1},\dots, g_{p}$, then $M$ can be naturally considered as
a filtered $A_{n}(K)$-module with the filtration $(M_{r})_{r\in\mathbb{Z}}$ where  $M_{r} = \sum_{i=1}^{p}W_{r}g_{i}$ for $r \in\mathbb{Z}$. It is clear
that each $M_{r}$ is a finitely generated $K$-vector space, $W_{r}M_{s} = M_{r+s}$ for all $r, s\in\mathbb{N}$, and $\bigcup_{r\in\mathbb{N}}M_{r} = M$.

\bigskip

The following result was first is proved in \cite{B1} (cf. {\cite[Chapter 1, Corollaries 3.3, 3.5, and Theorem 4.1]{Bj}).

\begin{thm}
With the above notation, there exists a polynomial $\psi_{M}(t)\in\mathbb{Q}[t]$ with the following properties.

\smallskip

{\em (i)}\, $\psi_{M}(r) = \dim_{K}M_{r}$ for all sufficiently large $r\in\mathbb{Z}$ (i.e., there exists $r_{0}\in\mathbb{Z}$ such that the last equality holds for all integers $r\geq r_{0}$);

\smallskip

{\em (ii)}\, $n\leq \deg\,\psi_{M}(t)\leq 2n$;

\smallskip

{\em (iii)}\, If $\psi_{M}(t) = a_{d}t^{d} +\dots + a_{1}t + a_{0}$ ($a_{d},\dots, a_{1}, a_{0}\in\mathbb{Q}$), then the degree $d$ of the polynomial $\psi(t)$ and the integer $d!a_{d}$ do not depend on the choice of the system of generators $g_{1},\dots, g_{p}$ of $M$. These numbers are denoted by $d(M)$ and $e(M)$, they are called the Bernstein dimension and multiplicity of the module $M$, respectively.
\end{thm}

The polynomial $\psi_{M}(t)$ is called the {\em Bernstein polynomial} of the $A_{n}(K)$-module $M$ associated with the given system of generators.
The family of all finitely generated left $A_{n}(K)$-modules $M$ with $d(M) = n$ is denoted by $\mathcal{B}_{n}$ and is called the {\em Bernstein class\/} of $A_{n}(K)$-modules. The elements of $\mathcal{B}_{n}$ are also called {\em holonomic $D$-modules}. The following statement (see \cite[Chapter 1, Propositions 5.2, 5.3 and Theorem 5.4]{Bj}) gives some properties of such modules.

\begin{thm}
{\em (i)}\, If $0\rightarrow M_{1}\rightarrow M_{2}\rightarrow M_{3}\rightarrow 0$ is an exact sequence of left $A_{n}(K)$-modules, then $M_{2}\in \mathcal{B}_{n}$ if and only if $M_{1}\in \mathcal{B}_{n}$ and
$M_{3}\in \mathcal{B}_{n}$.

\smallskip

{\em (ii)}\, If $M\in \mathcal{B}_{n}$, then $M$ has a finite length as a left $A_{n}(K)$-module. In fact, every strictly increasing sequence of $A_{n}(K)$-modules contains at most $e(M)$ terms.

\smallskip

{\em (iii)}\, If $M$ is any filtered $A_{n}(K)$-module with an ascending filtration $(M_{r})_{r\in {\bf Z}}$ and there exist positive integers $a$ and $b$ such that $\dim_{K}M_{r}\leq ar^{n} + b(r+1)^{n-1}$ for all
$r\in\mathbb{N}$, then $M\in \mathcal{B}_{n}$ and $e(M)\leq n!a$.
\end{thm}

\smallskip

\begin{center}

MULTIVARIATE NUMERICAL POLYNOMIALS OF SUBSETS OF $\mathbb{N}^{n}$

\end{center}

\smallskip

\medskip

\begin{defn}
A polynomial $f(t_{1}, \dots,t_{p})$ in $p$ variables ($p\geq 1$) with rational coefficients is said to be {\em numerical} if $f(t_{1},\dots, t_{p})\in\mathbb{Z}$ for all sufficiently large $t_{1}, \dots, t_{p}\in\mathbb{Z}$, that is, there exists $(s_{1},\dots, s_{p})\in\mathbb{Z}^{p}$ such that $f(r_{1},\dots, r_{p})\in\mathbb{Z}$ whenever $(r_{1},\dots, r_{p})\in\mathbb{Z}^{p}$ and $r_{i}\geq s_{i}$ ($1\leq i\leq p$).
\end{defn}
Clearly, every polynomial with integer coefficients is numerical.  As an example of a numerical polynomial in $p$ variables with non-integer coefficients one can consider $\prod_{i=1}^{p}{t_{i}\choose m_{i}}$ \, ($m_{1},\dots,
m_{p}\in\mathbb{Z}$), where ${t\choose k} =\frac{t(t-1)\dots (t-k+1)}{k!}$ for any $k\in\mathbb{Z}, k\geq 1$, ${t\choose0} = 1$, and ${t\choose k} = 0$ if $k$ is a negative integer.

If $f$ is a numerical polynomial in $p$ variables ($p > 1$), then $\deg f$ and $\deg_{t_{i}}f$ ($1\leq i\leq p$) will denote the total degree of $f$ and the degree of $f$ relative to the variable
$t_{i}$, respectively. The following theorem gives the ''canonical'' representation of a numerical polynomial in several variables.

\begin{thm}
Let $f(t_{1},\dots, t_{p})$ be a numerical polynomial in $p$ variables $t_{1},\dots,  t_{p}$, and let $\deg_{t_{i}}\, f = n_{i}$ ($1\leq i\leq p$). Then the polynomial $f(t_{1},\dots, t_{p})$ can
be represented as
\begin{equation}
f(t_{1},\dots t_{p}) =\D\sum_{i_{1}=0}^{n_{1}}\dots\D\sum_{i_{p}=0}^{n_{p}} {a_{i_{1}\dots i_{p}}}{t_{1}+i_{1}\choose i_{1}}\dots{t_{p}+i_{p} \choose i_{p}}
\end{equation}
\noindent with integer coefficients $a_{i_{1}\dots i_{p}}$ that are uniquely defined by the numerical polynomial.
\end{thm}

In the rest of this section we deal with subsets of $\mathbb{N}^{n}$ where the positive integer $n$ is represented as a sum of $p$ nonnegative integers $n_{1},\dots, n_{p}$ ($p \geq 1$).
In other words, we fix a partition $(n_{1},\dots, n_{p})$ of $n$.

If $A\subseteq\mathbb{N}^{n}$, then for any $r_{1},\dots, r_{p}\in\mathbb{N}$, $A(r_{1},\dots, r_{p})$ will denote the subset of $A$ that consists of all $n$-tuples $(a_{1},\dots , a_{n})\in A$ such that
$a_{1}+ \dots +a_{n_{1}}\leq r_{1}$, $a_{n_{1}+1} + \dots + a_{n_{1}+n_{2}}\leq r_{2}, \dots, a_{n_{1}+\dots +n_{p-1}+1} + \dots + a_{n}\leq r_{p}$. Furthermore, we shall associate with the set $A$
a set $V_{A}\subseteq\mathbb{N}^{m}$ that consists of all $n$-tuples $v = (v_{1},\dots , v_{n})\in\mathbb{N}^{n}$ that are not greater than or equal to any $n$-tuple from $A$ with respect to the product
order on $\mathbb{N}^{n}$. Thus, an element $v=(v_{1},\dots , v_{n})\in\mathbb{N}^{n}$ belongs to $V_{A}$ if and only if for any element $(a_{1},\dots , a_{n})\in A$ there exists $i\in\mathbb{N}, 1\leq i\leq n$, such that $a_{i} > v_{i}$.

The following two theorems proved in \cite[Chapter 2]{KLMP} generalize the well-known Kolchin's result on univariate numerical polynomials of subsets of $\mathbb{N}^{n}$ (see \cite[Chapter 0, Lemma 17]{K}) and give an explicit formula for the numerical polynomials in $p$ variables associated with a finite subset of $\mathbb{N}^{n}$.

\begin{thm}
Let $A$ be a subset of $\mathbb{N}^{n}$ where $n = n_{1} + \dots + n_{p}$ for some nonnegative integers  $n_{1},\dots, n_{p}$ ($p\geq 1$). Then there exists a numerical polynomial
$\omega_{A}(t_{1},\dots, t_{p})$ in $p$ variables with the following properties:

{\em (i)}  $\omega_{A}(r_{1},\dots, r_{p}) = \Card\,V_{A}(r_{1},\dots, r_{p})$ for all sufficiently large $(r_{1},\dots, r_{p})\in\mathbb{N}^{p}$ (i. e., there is
$(s_{1},\dots, s_{p})\in\mathbb{N}^{p}$ such that the equality holds for all $(r_{1},\dots, r_{p})\in\mathbb{N}^{p}$ such that $(s_{1},\dots, s_{p})\leq_{P}(r_{1},\dots, r_{p})$).

\smallskip

{\em (ii)} $\deg\omega_{A}\leq n$ and $\deg_{t_{i}}\omega_{A}\leq n_{i}$ for $i = 1,\dots, p$.

\smallskip

{\em (iii)} $\deg\,\omega_{A} = n$ if and only if the set $A$ is empty. In this case \\ $\omega_{A}(t_{1},\dots, t_{p}) = \D\prod_{i=1}^{p}{t_{i}+n_{i}\choose n_{i}}$.

\smallskip

{\em (iv)} $\omega_{A}$ is a zero polynomial if and only if $(0,\dots, 0)\in A$.
\end{thm}

\begin{defn}
The polynomial $\omega_{A}(t_{1},\dots, t_{p})$ is called the dimension polynomial of the set $A\subseteq\mathbb{N}^{n}$ associated with the partition $(n_{1},\dots, n_{p})$ of $n$.
\end{defn}

\begin{thm}
Let $A = \{a_{1}, \dots, a_{m}\}$ be a finite subset of $\mathbb{N}^{n}$ where $m$ is a positive integer and $n = n_{1} + \dots + n_{p}$ for some nonnegative integers  $n_{1},\dots, n_{p}$ ($p\geq 1$). Let
$a_{i} = (a_{i1}, \dots, a_{in})$ \, ($1\leq i\leq m$) and for any $l\in\mathbb{N}$, $0\leq l\leq m$, let $\Gamma(l,m)$ denote the set of all $l$-element subsets of the set $\mathbb{N}_{m} = \{1,\dots, m\}$. Furthermore, for any $\sigma \in\Gamma (l,m)$, let $\bar{a}_{\sigma j} = \max \{a_{ij} | i\in\sigma\}$ ($1\leq j\leq n$) and $b_{\sigma j} =\D\sum_{h\in\sigma_{j}}\bar{a}_{\sigma h}$. Then
\begin{equation}
\omega_{A}(t_{1},\dots, t_{p}) = \D\sum_{l=0}^{m}(-1)^{l}\D\sum_{\sigma \in \Gamma (l,m)}\D\prod_{j=1}^{p}{t_{j}+n_{j} - b_{\sigma j}\choose n_{j}}
\end{equation}
\end{thm}

\begin{rem}  Clearly, if $A\subseteq\mathbb{N}^{n}$ and $A'$ is the set of all minimal elements of the set $A$ with respect to the product order on  $\mathbb{N}^{n}$, then the set $A'$ is finite and
$\omega_{A}(t_{1}, \dots, t_{p}) = \omega_{A'}(t_{1}, \dots, t_{p})$. Thus, Theorem 2.7 gives an algorithm that allows one to find the dimension polynomial of any subset of $\mathbb{N}^{n}$ (with a given representation of $n$ as a sum of $p$ positive integers): one should first find the set of all minimal points of the subset and then apply Theorem 2.7.
\end{rem}

\section{Reduction with respect to several term orderings and generalized Gr\"obner bases in free $D$-modules}

In what follows, we keep the notation and conventions of the previous section. (In particular, $A_{n}(K)$ denotes a Weyl algebra in $n$ variables $x_{1},\dots, x_{n}$ over a field $K$ of characteristic zero, and the corresponding partial derivations are denoted by $\partial_{1},\dots, \partial_{n}$, respectively.) Let us fix a partition of the  set $X = \{x_{1},\dots, x_{n}\}$ into $p$ disjoint subsets and the corresponding partition of the set $\Delta = \{\partial_{1},\dots, \partial_{n}\}$:
\begin{equation}
X = X_{1}\bigcup\dots\bigcup X_{p}\,\,\,\,\, \text{and}\,\,\,\,\, \Delta = \Delta_{1}\bigcup \dots \bigcup \Delta_{p}
\end{equation}
where $X_{1} = \{x_{1},\dots, x_{n_{1}}\},\dots, X_{p} = \{x_{n_{1}+\dots+n_{p-1}+1},\dots, x_{n}\}$ and \\ $\Delta_{1} = \{\partial_{1},\dots, \partial_{n_{1}}\},\dots,
\Delta_{p} = \{\partial_{n_{1}+\dots+n_{p-1}+1},\dots, \partial_{n}\}$  ($n_{1}+\dots + n_{p} = n$).

\noindent Furthermore, $\Theta$ will denote the set of all power products $\theta=x_{1}^{\alpha_{1}}\dots x_{n}^{\alpha_{n}}\partial_{1}^{\beta_{1}}\dots \partial_{n}^{\beta_{n}}$ with nonnegative integer exponents; such a power product will be called a {\em monomial} and also denoted by $x^{\alpha}\partial^{\beta}$  (we use the multi-index notation introduced in the previous section).

Considering a partition $\mathbb{N}_{n}= N_{1}\bigcup\dots\bigcup N_{p}$ of the set $\mathbb{N}_{n} = \{1,\dots, n\}$ that corresponds to partition (3.1)\, (that is, $N_{1}= \{1,\dots, n_{1}\},\dots, N_{p}=\{n_{1}+\dots+n_{p-1}+1,\dots, n\}$) and a  monomial $\theta=x^{\alpha}\partial^{\beta} = x_{1}^{\alpha_{1}}\dots x_{n}^{\alpha_{n}}\partial_{1}^{\beta_{1}}\dots \partial_{n}^{\beta_{n}}$, we set
$\theta_{xj} = \prod_{k\in N_{j}}x_{k}^{\alpha_{k}}$, $\theta_{\partial j} = \prod_{k\in N_{j}}\partial_{k}^{\beta_{ k}}$, $\theta_{x} = x^{\alpha} = x_{1}^{\alpha_{1}}\dots x_{n}^{\alpha_{n}}$, and
$\theta_{\partial} = \partial^{\beta} = \partial_{1}^{\beta_{1}}\dots \partial_{n}^{\beta_{n}}$. Furthermore, we set
$|\theta|_{x} = \sum_{i=1}^{n}\alpha_{i}$, $|\theta|_{\partial} = \sum_{i=1}^{n}\beta_{i}$, $|\theta|_{xj} = \sum_{i\in N_{j}}\alpha_{i}$, and $|\theta|_{\partial j} = \sum_{i\in N_{j}}\beta_{i}$
($1\leq j\leq p$). With this notation, the numbers $\ord\theta = |\theta|_{x} + |\theta|_{\partial}$ and $\ord_{j}\theta = |\theta|_{xj} + |\theta|_{\partial j}$ will be called the {\em order of $\theta$} and the {\em $j$th order of $\theta$} ($1\leq j\leq p$), respectively. If $r_{1},\dots, r_{p}\in\mathbb{N}$, the set $\{\theta \in \Theta\,|\,\ord_{j}\theta \leq r_{j}$ for $j=1,\dots, p\}$ will be denoted by $\Theta(r_{1},\dots, r_{p})$.

It is easy to see that the sets $\Theta_{x} = \{\theta_{x}\,|\theta\in\Theta\}$ and $\Theta_{\partial} = \{\theta_{\partial}\,|\theta\in\Theta\}$ are commutative multiplicative semigroups (of course, $\Theta$ is not: $\partial_{i}x_{i}\neq x_{i}\partial_{i}$ for $i=1,\dots, n$).

If $D = \sum_{\theta\in\Theta}a_{\theta}\theta\in A_{n}(K)\setminus\{0\}$ ($a_{\theta}\in K$ for any $\theta\in\Theta$ and the sum is finite), then the order $\ord D$ and the $j$th order $\ord_{j}D$ ($1\leq j\leq p$) of $D$ are defined, as follows: $\ord D = \max\{\ord\theta\,|\, a_{\theta}\neq 0\}$ and $\ord_{j}D = \max\{\ord_{j}\theta\,|\, a_{\theta}\neq 0\}$.
These notions allow one to consider the Weyl algebra $A_{n}(K)$ as a filtered ring with a $p$-dimensional filtration $\{W_{r_{1},\dots, r_{p}}\,|\,(r_{1},\dots, r_{p})\in\mathbb{Z}^{p}\}$
where $\{W_{r_{1},\dots, r_{p}} = \{D\in A_{n}(K)\,|\,\ord_{j}D\leq r_{j}$ for $j=1,\dots, p\}$  for any $(r_{1},\dots, r_{p})\in\mathbb{N}^{p}$ and $W_{r_{1},\dots, r_{p}} = 0$ for all
$(r_{1},\dots, r_{p})\in\mathbb{Z}^{p}\setminus\mathbb{N}^{p}$. Clearly, $\bigcup{\{W_{r_{1},\dots, r_{p}}\,|\,r_{1},\dots, r_{p})\in\mathbb{Z}^{p}\}}=A_{n}(K)$, $W_{r_{1}, \dots, r_{p}}\subseteq W_{s_{1},\dots, s_{p}}$ if $(r_{1}, \dots, r_{p})\leq_{P}(s_{1}, \dots, s_{p})$ ($\leq_{P}$ is the product order on $\mathbb{Z}^{p}$) and if $(i_{1}, \dots, i_{p}), (j_{1}, \dots, j_{p})\in\mathbb{N}^{p}$, then
$W_{i_{1}, \dots, i_{p}}W_{j_{1},\dots, j_{p}} = W_{i_{1}+j_{1}, \dots, i_{p}+j_{p}}$.

\begin{defn}
If $M$ is a $D$-module (that is, a left $A_{n}(K)$-module), then a family \\$\{M_{r_{1}, \dots, r_{p}}\,\,(r_{1},\dots, r_{p})\in\mathbb{Z}^{p}\}$ of $K$-vector subspaces of $M$ is said to be a $p$-dimensional
filtration of $M$ if the following four conditions hold.

\smallskip

(i) $M_{r_{1}, \dots, r_{p}}\subseteq M_{s_{1}, \dots, s_{p}}$ for any $p$-tuples $(r_{1}, \dots, r_{p}), \,(s_{1}, \dots, s_{p})\in\mathbb{Z}^{p}$ such that \\$(r_{1}, \dots, r_{p})\leq_{P}(s_{1}, \dots, s_{p})$.

\smallskip

(ii) $\D\bigcup_{(r_{1}, \dots, r_{p})\in\mathbb{Z}^{p}}M_{r_{1}, \dots, r_{p}} = M$.

\smallskip

(iii) There exists a $p$-tuple $(r^{(0)}_{1},\dots, r^{(0)}_{p})\in\mathbb{Z}^{p}$ such that $M_{r_{1}, \dots, r_{p}} = 0$ if $r_{i} < r^{(0)}_{i}$ for at least one index $i$ ($1\leq i\leq p$).

\smallskip

(iv) $W_{r_{1}, \dots, r_{p}}M_{s_{1}, \dots, s_{p}}\subseteq M_{r_{1}+s_{1}, \dots, r_{p}+s_{p}}$ for all $(r_{1}, \dots, r_{p}), (s_{1}, \dots, s_{p})\in\mathbb{Z}^{p}$.

If every vector $K$-space $M_{r_{1}, \dots, r_{p}}$ is finite-dimensional and there exists an element $(h_{1}, \dots, h_{p})$ in $\mathbb{Z}^{p}$ such that $W_{r_{1}, \dots, r_{p}}M_{h_{1}, \dots, h_{p}} = M_{r_{1}+h_{1}, \dots, r_{p}+h_{p}}$ for any $(r_{1}, \dots, r_{p})\in\mathbb{N}^{p}$, the $p$-dimensional filtration is called excellent.
\end{defn}

Clearly, if $z_{1},\dots, z_{k}$ is a finite system of generators of a $D$-module $M$ over $A_{n}(K)$, then $\{\D\sum_{i=1}^{k}W_{r_{1}, \dots, r_{p}}z_{i}\,|\,(r_{1},\dots, r_{p})\in\mathbb{Z}^{p}\}$  is an excellent $p$-dimensional filtration of $M$.

\smallskip

We shall consider $p$ orderings $<_{i}$ ($1\leq i\leq p$) of the set of all monomials $\Theta$ defined as follows: if $\theta = x^{\alpha}\partial^{\beta} = x_{1}^{\alpha_{1}}\dots x_{n}^{\alpha_{n}}\partial_{1}^{\beta_{1}}\dots \partial_{n}^{\beta_{n}}$ and $\theta' = x^{\gamma}\partial^{\delta} = x_{1}^{\gamma_{1}}\dots x_{n}^{\gamma_{n}}\partial_{1}^{\delta_{1}}\dots \partial_{n}^{\delta_{n}}$ are two elements of $\Theta$, then $\theta <_{i} \theta'$ if and only if the vector

\noindent$(\ord_{i}\theta, \ord_{1}\theta,\dots, \ord_{i-1}\theta, \ord_{i+1}\theta, \dots, \ord_{p}\theta, \alpha_{n_{1}+\dots + n_{i-1}+1}, \dots, \alpha_{n_{1}+\dots +n_{i}},\\
\beta_{n_{1}+\dots + n_{i-1}+1}, \dots, \beta_{n_{1}+\dots +n_{i}}, \alpha_{1},\dots, \beta_{1},\dots, \alpha_{n_{1}+\dots + n_{i-1}}, \dots, \beta_{n_{1}+\dots + n_{i-1}},\\
\alpha_{n_{1}+\dots + n_{i+1}}, \dots, \beta_{n_{1}+\dots + n_{i+1}}, \dots, \alpha_{n}, \beta_{n})$ is less than the corresponding vector for $\theta'$ with respect to the lexicographic order on $\mathbb{Z}^{2n+p}$
(and with the obvious adjustment for $i=1$). Clearly, $\Theta$ is well-ordered with respect to each of the orders $<_{1}, \dots, <_{p}$.

Let $\theta = \theta_{x}\theta_{\partial}, \theta' = \theta'_{x}\theta'_{\partial}\in\Theta$.  We say that $\theta$ divides $\theta'$ and write $\theta|\theta'$ if $\theta_{x}|\theta'_{x}$ and $\theta_{\partial}|\theta'_{\partial}$ in the commutative semigroups $\Theta_{x}$ and $\Theta_{\partial}$, respectively. In this case we also say that $\theta'$ is a {\em multiple} of $\theta$.

It is easy to see that if $\theta\,|\,\theta'$, then there exist elements $\theta_{0}, \theta_{1},\dots, \theta_{k}\in \Theta$ such that $\theta' = \theta_{0}\theta - \sum_{i=1}^{k}\theta_{i}$ where
$\ord_{j}\theta_{0} + \ord_{j}\theta = \ord_{j}\theta'$ (and also $|\theta_{0}|_{xj} + |\theta|_{xj} = |\theta'|_{xj}$ and $|\theta_{0}|_{\partial j} + |\theta|_{\partial j} = |\theta'|_{\partial j}$)
for $j=1,\dots, p$, while for any $i=1,\dots, k$, one has $\ord\theta_{i} < \ord\theta'$, $\ord_{j}\theta_{i}\leq\ord_{j}\theta'$ ($1\leq j\leq p$), and $\ord_{j_{0}}\theta_{i} < \ord_{j_{0}}\theta'$ for some
$j_{0}\in\{1,\dots, p\}$. (It follows that $\theta_{i} <_{j}\theta'$ for $j=1,\dots, p$.)
In this case we denote the monomial $\theta_{0}$ by ${\D\frac{\theta'}{\theta}}$.

For example ($n=2$), $\theta = x_{1}x_{2}\partial_{2}$ divides $\theta' = x_{1}^{2}x_{2}\partial_{1}\partial_{2}^{2}$  and
one can write $\theta' = \theta_{0}\theta - (\theta_{1}+\theta_{2}+\theta_{3})$ where $\theta_{0} = x_{1}\partial_{1}\partial_{2}$, $\theta_{1} = x_{1}^{2}\partial_{1}\partial_{2}$, $\theta_{2}= x_{1}x_{2}\partial_{2}^{2}$, and $\theta_{3} = x_{1}\partial_{2}$.

In what follows, by the {\em least common multiple} of elements $\theta', \theta''\in \Theta$ we mean the element $\lcm(\theta', \theta'') = \lcm(\theta_{x}', \theta_{x}'')\lcm(\theta_{\partial}', \theta_{\partial}'')$. Clearly, if $\theta = \lcm(\theta', \theta'')$, then $\theta'|\theta$,\, $\theta''|\theta$, and whenever $\theta'|\tau$ and $\theta''|\tau$ for some $\tau\in \Theta$, one has $\theta|\tau$.

Let $E$ be a finitely generated free $A_{n}(K)$-module with set of free generators $e=\{e_{1},\dots, e_{m}\}$. Obviously, $E$ can be considered as a $K$-vector space with the basis $\Theta e = \{\theta e_{i} | \theta \in \Theta, 1\leq i\leq m\}$ whose elements will be called {\em terms}. For any term $\theta e_{i}$ ($\theta\in \Theta,\, 1\leq i\leq m$) and for any $j\in\{1,\dots, p\}$, we define the $j$th order $\ord_{j}(\theta e_{j})$ of this term as $\ord_{j}\theta $. If $T\subseteq \Theta$, then the set $\{te_{i}\,|\,t\in T,\, 1\leq i\leq m\}$ will be denoted by $Te$. In particular, for any $r_{1}, \dots, r_{p}\in\mathbb{N}^{p}$, $\Theta(r_{1},\dots, r_{p})e$ will denote the set $\{u=\theta e_{i}\,|\,\ord_{j}u\leq r_{j}$ for $j=1,\dots, p\}$.

A term $u = \theta'e_{i}$ is said to be a {\em multiple} of a term $v = \theta e_{j}$ ($\theta, \theta'\in\Theta,\, 1\leq i, j\leq m$) if $i=j$ and $\theta|\theta'$. In this case we also say that $v$ divides $u$, write $v|u$ and define ${\D\frac{u}{v}}$ as $ {\D\frac{\theta'}{\theta}}.$

The {\em least common multiple} of two terms $w_{1} = \theta_{1}e_{i}$ and $w_{2} = \theta_{2}e_{j}$ is defined as $$\lcm(w_{1}, w_{2}):=
\begin{cases}
\lcm(\theta_1,\theta_2)e_i,&\textnormal{ if }i=j,\\
0,&\textnormal{ if }i\neq j.
\end{cases}$$
We shall consider $p$ orderings of the set $\Theta e$ that correspond to the introduced $p$ orderings of the set $\Theta$ (and are denoted by the same symbols $<_{1},\dots, <_{p}$):
$\theta e_{i} <_{j}\theta' e_{k}$ if and only if $\theta <_{j}\theta'$ or $\theta=\theta'$ and $i < k$.

Clearly, the set of all terms $\Theta e$ is a basis of the $K$-vector space $E$, so every nonzero element $f\in E$ has a unique representation of the form
\begin{equation}
f = a_{1}\theta_{1}e_{i_{1}} + \dots + a_{s}\theta_{s}e_{i_{s}}
\end{equation}
where $\theta_{j}\in\Theta$, $0\neq a_{j}\in K$ ($1\leq j\leq s$) and the terms $\theta_{1}e_{i_{1}},\dots, \theta_{s}e_{i_{s}}$ are all distinct. We say that a term $u$ {\em appears in $f$} (or that $f$ {\em contains} $u$) if $u$ is one of the terms $\theta_{k}e_{i_{k}}$ in the representation (3.2) (that is, the coefficient of $u$ in $f$ is not zero).
\begin{defn}
Let an element $0\neq f\in E$ be written in the form (3.2). Then the greatest term of the set $\{\theta_{1}e_{i_{1}}, \dots, \theta_{s}e_{i_{s}}\}$ with respect to the order $<_{j}$ ($1\leq j\leq p$) is called the $j$th leader of $f$; it is denoted by $u_{f}^{(j)}$.   Furthermore, $\lc_{j}(f)$ will denote the coefficients of $u_{f}^{(j)}$ in (3.2).
\end{defn}
\begin{defn}
Let $f, g\in E$ and let $k, i_{1}, \dots, i_{r}$ be distinct elements of the set $\{1,\dots, p\}$.  Then $f$ is said to be $(<_{k}, <_{i_{1}}, \dots <_{i_{r}})$-{\bf reduced} with respect to $g$ if $f$ does not contain any multiple $\theta u_{g}^{(k)}$  $(\theta \in \Theta)$ such that $\ord_{i_{\nu}}(\theta u_{g}^{(i_{\nu})})\leq \ord_{i_{\nu}}u_{f}^{(i_{\nu})}$ for $\nu = 1,\dots, r$. An element $f\in E$ is said to be
$(<_{k}, <_{i_{1}}, \dots <_{i_{r}})$-reduced with respect to a set $G\subseteq E$, if $f$ is $(<_{k}, <_{i_{1}}, \dots <_{i_{r}})$-reduced with respect to every element of $G$.
\end{defn}

With the above notation, let us consider $p-1$ new symbols $z_{1},\dots, z_{p-1}$ and a set $\Gamma$ of all power products of the form $\gamma = \theta z_{1}^{k_{1}}\dots z_{p-1}^{k_{p-1}}$ with $\theta\in\Theta$ and $l_{1},\dots, l_{p-1}\in\mathbb{N}$. If $\gamma' = \theta'z_{1}^{l_{1}}\dots z_{p-1}^{l_{p-1}}\in\Gamma$, we say that $\gamma$ divides $\gamma'$ and write $\gamma|\gamma'$ if $\theta|\theta'$ (in the sense of division in $\Theta$) and $k_{i}\leq l_{i}$ for $i=1,\dots, p-1$. Furthermore, let $\Gamma e$ denote the set of all elements of the form $\gamma e_{i}$ where $\gamma\in\Gamma$, $1\leq i\leq m$. If $t= \gamma e_{i}, t'=\gamma' e_{k}\in\Gamma e$, we say that $t$ divides $t'$ (or that $t'$ is a multiple of $t$) if and only if $i=k$ and $\gamma|\gamma'$.

For any element $f\in E$, we set $d_{i}(f) = \ord_{i}u_{f}^{(i)} - \ord_{i}u_{f}^{(1)}$ ($2\leq i\leq p$) and consider a mapping $\rho: E\rightarrow \Gamma e$ be defined by
$\rho(f) = z_{1}^{d_{2}(f)}\dots z_{p-1}^{d_{p(f)}}u_{f}^{(1)}$.
\begin{defn}
With the above notation, let $N$ be an $A_{n}(K)$-submodule of $E$. A finite set $G = \{g_{1},\dots, g_{r}\}\subseteq N$ is called a Gr\"obner basis of $N$ with respect to the orders $<_{1},\dots, <_{p}$ if for any $f\in N$, there exists $g_{i}\in G$ such that $\rho(g_{i}) | \rho(f)$.
\end{defn}
\begin{rem} The above condition $\rho(g_{i})\,|\,\rho(f)$ means that $f$ is not $(<_{1},\dots, <_{p})$-reduced with respect to $g_{i}$, since the inequality $d_{j}(g_{i})\leq d_{j}(f)$ for $j=2,\dots, p$ means
that ${\D\frac{u_{f}^{(1)}}{u_{g_{i}}^{(1)}}} = \theta\in\Theta$ such that $\ord_{j}(\theta u_{g_{i}}^{(j)}) = \ord_{j}\theta + \ord_{j}u_{g_{i}}^{(j)}\leq \ord_{j}u_{f}^{(j)}$ ($2\leq j\leq p$). The expression of this property with the use of the divisibility of elements of $\Gamma e$ is convenient because it also shows that the existence of the Gr\"obner basis in the sense of Definition 3.4 immediately follows from the Dickson's lemma. Furthermore, it is clear that a Gr\"obner basis of $N\subseteq E$ with respect to the orders $<_{1},\dots, <_{p}$ is a  Gr\"obner basis of $N$ with respect to $<_{1}$ in the usual sense and therefore generates $N$ as an $A_{n}(K)$-module.
\end{rem}
A set $\{g_{1},\dots, g_{r}\}\subseteq E$ is said to be a {\em Gr\"obner basis with respect to the orders $<_{1},\dots,<_{p}$} if $G$ is a Gr\"obner basis of $N = \sum_{i=1}^{r}A_{n}(K)g_{i}$.
with respect to $<_{1},\dots, <_{p}$.

For any $f, g, h\in E$, with $g\neq 0$, we say that the element $f$ $(<_{k}, <_{i_{1}}, \dots, <_{i_{l}})$-reduces to $h$ modulo $g$ in one step and write
$f\xrightarrow[<_{k}, <_{i_{1}}, \dots, <_{i_{l}}] {\text{g}} h$ if and only if $u_{g}^{(k)} | w$ for some term $w$ in $f$ with a coefficient $a$, \, ${\D\frac{w}{u_{g_{i}}^{(k)}}} = \theta\in\Theta$,
$h = f - a\lc_{k}(g)^{-1}\theta g$\, and $\ord_{i_{\nu}}\theta + \ord_{i_{\nu}}u^{(i_{\nu})}_{g} \leq \ord_{i_{\nu}}u_{f}^{(i_{\nu})}$ ($1\leq \nu \leq l$).

If $f, h\in E$ and $G\subseteq E$, then we say that the element $f$ $(<_{k}, <_{i_{1}}, \dots, <_{i_{l}})$-reduces to $h$ modulo $G$ and write $f\xrightarrow[<_{k}; <_{i_{1}}, \dots, <_{i_{l}}] {\text{$G$}} h$ if and only if there exist two sequences\, $g^{(1)}, g^{(2)},\dots, g^{(q)}\in G$ and  $h_{1},\dots, h_{q-1}\in F$ such
that

\noindent$f\xrightarrow[<_{k}, <_{i_{1}}, \dots, <_{i_{l}}]{\text{$g^{(1)}$}} h_{1}\xrightarrow[<_{k}, <_{i_{1}}, \dots, <_{i_{l}}] {\text{$g^{(2)}$}} \dots \xrightarrow[<_{k}, <_{i_{1}}, \dots, <_{i_{l}}] {\text{$g^{(q-1)}$}}h_{q-1}\xrightarrow[<_{k}, <_{i_{1}}, \dots, <_{i_{l}}] {\text{$g^{(q)}$}}h$.

\smallskip

Let $G = \{g_{1},\dots, g_{r}\}$ be a finite set of elements in the free $A_{n}(K)$-module $E$. Then the following algorithm (whose termination is obvious) shows that for every $f\in E$ there exist elements $g\in E$ and $Q_{1},\dots,Q_{r}\in A_{n}(K)$ such that $f-g=\D\sum_{i=1}^{r}Q_{i}g_{i}$ and $g$ is $(<_{1},\dots, <_{p})$-reduced with respect to $G$.

\begin{alg}

($f, r, g_{1},\dots, g_{r};\, g;\, Q_{1},\dots, Q_{r}$)

\smallskip

{\bf Input:} $f\in E$, {\em a positive integer} $r$, $G = \{g_{1},\dots, g_{r}\}\subseteq E$ {\em where} $g_{i}\ne 0$ {\em for} $i = 1,\dots, r$

{\bf Output:} {\em An element} $g\in E$ {\em and elements} $Q_{1},\dots, Q_{r}\in A_{n}(K)$ {\em such that}
$g = f - (Q_{1}g_{1} + \dots + Q_{r}g_{r})$\, {\em and} $g$ {\em is $(<_{1},\dots, <_{p})$-reduced with respect to} $G$

{\bf Begin}

$Q_{1}:= 0, \dots, Q_{r}:= 0, g:= f$

{\bf While} {\em there exist} $i$, $1\leq i\leq r$, {\em and a term}
$w$, {\em that appears in} $g$ {\em with a nonzero coefficient} $c(w)$, {\em such that}
$u^{(1)}_{g_{i}}|w$ {\em and} $\ord_{j}(\frac{w}{u^{(1)}_{g_{i}}}u^{(j)}_{g_{i}})\leq
\ord_{j}u^{(j)}_{g}$ for $j=2,\dots, p$ {\bf do}

$z$:= {\em the greatest (with respect to} $<_{1}$)
{\em of the terms} $w$ {\em that satisfies the above conditions.}

$k$:= {\em the smallest number} $i$ {\em for which} $u^{(1)}_{g_{i}}$ {\em is the greatest (with
respect to} $<_{1}$) $1${\em -leader of an element} $g_{i} \in G$ {\em such that}
$u^{(1)}_{g_{i}}|z$ {\em and}
$\ord_{j}(\frac{z}{u^{(1)}_{g_{i}}}u_{g_{i}}) \leq
\ord_{j}u^{(j)}_{g}$ for $j=2,\dots, p$.

$Q_{k}:= Q_{k} + c(z)\lc_{1}(g_{k})^{-1}{\frac{z}{u_{g_{k}}^{(1)}}}\,g_{k}$

$g:= g - c(z)\lc_{1}(g_{k})^{-1}{\frac{z}{u_{g_{k}}^{(1)}}}\,g_{k}$

{\bf End}

\end{alg}

Now, for any nonzero elements $f, g\in F$, we can define the {\em $r$th $S$-polynomial of $f$ and $g$} ($1\leq r\leq p$) as the element
$$S_{r}(f, g) = \lc_{r}(f)^{-1}\D\frac{\lcm(u_{f}^{(r)}, u_{g}^{(r)})}{u_{f}^{(r)}}f - \lc_{r}(g)^{-1}\D\frac{\lcm(u_{f}^{(r)}, u_{g}^{(r)})}{u_{g}^{(r)}}\,g.$$

The following two statements can be obtained by mimicking the proof of the corresponding statements in \cite[Section 4]{Levin1} (Proposition 3.7 and Theorem 3.10).

\begin{prop}
Let $G = \{g_{1},\dots, g_{t}\}$ be a Gr\"obner basis of an $A_{n}(K)$-submodule $N$ of $E$ with respect to the orders $<_{1},\dots, <_{p}$. Then

{\em (i)}\, $f\in N$ if and only if $f\xrightarrow[<_{1}, <_{2}, \dots <_{p}]{\text{$G$}} 0$\,.

{\em (ii)}\,  If $f\in N$ and $f$ is $(<_{1}, <_{2}, \dots <_{p})$-reduced with respect to $G$, then $f=0$.
\end{prop}

\begin{prop}
Let $G = \{g_{1},\dots, g_{t}\}$ be a Gr\"obner basis of an $A_{n}(K)$-submodule $N$ of $E$ with respect to each of the following sequences of orders: $<_{p}$;\, $<_{p-1}, <_{p}$;\, \dots
; $<_{r+1},\dots, <_{p}$ ($1\leq r\leq p-1$). Furthermore, suppose that

\noindent$S_{r}(g_{i}, g_{j}) \xrightarrow[<_{r}, <_{r+1}, \dots <_{p}]{\text{$G$}} 0$ for any $g_{i}, g_{j}\in G$.

\noindent Then $G$ is a Gr\"obner basis of $N$ with respect to the sequence of orders $<_{r}, <_{r+1}, \dots, <_{p}$.
\end{prop}
Note that the last proposition, together with Algorithm 1, gives an algorithm for the computation of Gr\"obner bases of a submodule of a finitely generated free $A_{n}(K)$-module with respect to a sequence of several term orderings. This algorithm (together with Algorithm 3.1) is currently being implemented in MAPLE and PYTHON.

\begin{ex}
 Let $E$ be a free $A_2(K)$ module with free generators $e_1,e_2$. Consider the $A_n(K)$-submodule $N$ of $E$ generated by elements $h_{1}= x_{2}^{2}\partial_{1}\partial_{2}e_{2} + x_{1}x_{2}\partial_{1}\partial_{2}e_{1}$
 and $h_{2} = x_{2}\partial_{1}^{2}e_{2} + x_{1}\partial_{1}^{2}e_{1}$. With the above notation, $u_{h_{1}}^{(2)} = x_{2}^{2}\partial_{1}\partial_{2}e_{2}$, $u_{h_{2}}^{(2)} = x_{2}\partial_{1}^{2}e_{2}$ and
 $$S_{2}(h_{1}, h_{2}) = \partial_{1}h_{1} - x_{2}\partial_{2}h_{2} = x_{2}\partial_{1}\partial_{2}e_{1} - x_{2}\partial_{1}^{2}e_{2}.$$ This element is not $<_{2}$-reduced with respect to $\{h_{1}, h_{2}\}$, so we set
 $$h_{3} = x_{2}\partial_{1}\partial_{2}e_{1} - x_{2}\partial_{1}^{2}e_{2}; \,u_{h_{3}}^{(2)} = x_{2}\partial_{1}\partial_{2}e_{1}.$$ It is easy to see that $S_{2}(h_{1}, h_{3}) = S_{2}(h_{2}, h_{3}) = 0$, so
 $\{h_{1}, h_{2}, h_{3}\}$ is a Gr\"obner basis of $N$ with respect to $<_{2}$.

We have $u_{h_{1}}^{(1)} = x_{1}x_{2}\partial_{1}\partial_{2}e_{1}$, $u_{h_{2}}^{(1)} = x_{1}\partial_{1}^{2}e_{1}$, and $u_{h_{3}}^{(1)} = x_{2}\partial_{1}^{2}e_{2}$. Therefore,
$$S_{1}(h_{1}, h_{2}) = \partial_{1}h_{1} - x_{2}\partial_{2}h_{2} = h_{3}\xrightarrow[<_{1}, <_{2}]{\text{$h_{3}$}} 0.$$ Also, $S_{1}(h_{1}, h_{3}) = S_{1}(h_{2}, h_{3}) = 0$, since
$\lcm(u_{h_{1}}^{(1)}, u_{h_{3}}^{(1)}) = \lcm(u_{h_{2}}^{(1)}, u_{h_{3}}^{(1)}) = 0$. Thus, $\{h_{1}, h_{2}, h_{3}\}$ is a Gr\"obner basis of $N$ with respect to $<_{1}, <_{2}$.

Note that this is not a reduced Gr\"obner basis of $N$ with respect to $<_{1}$, since $h_{2}$ contains the $1$-leader of $h_{3}$. However, $\ord_{2}(1\cdot u_{h_{3}}^{(2)}) = 2 > \ord_{2}u_{h_{2}}^{(2)} = 1$, so
$h_{2}$ is ($<_{1}, <_{2}$)-reduced with respect to $h_{3}$.
\end{ex}

\section{Multivariate Bernstein-type Polynomials}

In this section we consider the Weyl algebra $A_{n}(K)$ as a ring with a $p$-dimensional filtration \\$\{W_{r_{1},\dots, r_{p}}\,|\,(r_{1},\dots, r_{p})\in\mathbb{Z}^{p}\}$ introduced in the previous section.  Recall that  $W_{r_{1},\dots, r_{p}}=0$, if at least one of the numbers $r_{i}$ is negative, and if $r_{i}\geq 0$ for $i=1,\dots, p$, then $W_{r_{1},\dots, r_{p}}$ is a $K$-vector space with a basis $\Theta(r_{1},\dots, r_{p})$. It follows from the third statement of Theorem 2.5 (taking into account that the set $X\bigcup\Delta$ in (3.1) is the disjoint union of $p$ $2n_{i}$-element sets $X_{i}\bigcup\Delta_{i}$, $1\leq i\le p$) that for any $r_{1},\dots, r_{p}\in\mathbb{N}$, we have
$$\dim_{K}W_{r_{1},\dots, r_{p}} = \Card\,\Theta(r_{1},\dots, r_{p}) = \prod_{i=1}^{p}{r_{i}+2n_{i}\choose 2n_{i}}.$$

In what follows we use properties of Gr\"obner bases with respect to several term-orderings (that correspond to the partition (3.1) of the set of generators of $A_{n}(K)$) to prove the existence and determine
invariants of multivariate dimension polynomials of finitely generated $A_{n}(K)$-modules. The following result is the main step in this direction (we use the notation and conventions of the previous section).

\begin{thm}
Let $M$ be a left $A_{n}(K)$-module generated by a finite set $\{f_{1},\dots, f_{m}\}$, and let $E$ be a free left $A_{n}(K)$--module with free generators $e_{1}, \dots, e_{m}$. Let $\pi: E\longrightarrow M$ be the natural
$A_{n}(K)$-epimorphism ($\pi(e_{i}) = f_{i}$ for $i=1, \dots, m$), $N = \Ker\pi$, and $G = \{g_{1}, \dots, g_{d}\}$ a Gr\"obner basis of the $A_{n}(K)$-module $N$ with respect to $<_{1},\dots, <_{p}$. Furthermore, for any
$(r_{1},\dots, r_{p})\in\mathbb{N}^{p}$, let $M_{r_{1},\dots, r_{p}} = \D\sum_{i=1}^{m}W_{r_{1},\dots, r_{p}}f_{i}$, \,\,$V_{r_{1},\dots, r_{p}} = \{u\in \Theta e\, |\, \ord_{i}u\leq r_{i}$ \,($1\leq i\leq
p$) and $u_{g}^{(1)}\nmid u$ for any $g\in G\}$,\,\, $V'_{r_{1},\dots, r_{p}} = \{u\in \Theta(r_{1},\dots, r_{p})e \setminus V_{r_{1},\dots, r_{p}}\,|\,$ whenever $u_{g}^{(1)}|u$ for some $g\in G$ and
$\theta = {\D\frac{u}{u_{g}^{(1)}}}$, there exists $i\in \{2,\dots, p\}$ such that $\ord_{i}\theta u_{g}^{(i)} > r_{i}\}$, and\\ $U_{r_{1}\dots r_{p}} = V_{r_{1},\dots, r_{p}}\bigcup V'_{r_{1},\dots, r_{p}}$. \,
Then for any $(r_{1}, \dots, r_{p})\in\mathbb{N}^{p}$, the set $\pi(U_{r_{1},\dots, r_{p}})$ is a basis of $M_{r_{1},\dots, r_{p}}$ over $K$.
\end{thm}

\begin{proof}
First let us prove that the set  $\pi(U_{r_{1}\dots r_{p}})$ is linearly independent over $K$.  Suppose that \\$\sum_{i=1}^{k}{a_{i}\pi(u_{i})} = 0$ for some $u_{1},\dots, u_{k}\in U_{r_{1}\dots r_{p}},\, a_{1},\dots, a_{k}\in K$. Then the element $h=\sum_{i=1}^{k}{a_{i}u_{i}}\in N$ is reduced with respect to $G$. Indeed, if a term $u=\theta e_{j}$ appears in $h$ (so that $u=u_{i}$ for some $i=1,\dots, k$), then either $u$ is not a multiple of any
$u_{g_{k}}^{(1)}$ ($1\leq k\leq d$) or $u=\theta u_{g_{k}}^{(1)}$ for some $\theta \in \Theta,\, 1\leq k \leq d$, such that $\ord_{j}(\theta u_{g_{k}}^{(j)}) > r_{j}$ for some $j\in\{2,\dots, p\}$.
By Proposition 3.6, $h = 0$, whence $a_{1}= \dots =a_{k} = 0$.

Now we are going to prove that the set $\pi(U_{r_{1},\dots, r_{p}})$ generates $M_{r_{1},\dots, r_{p}}$ as a vector space over $K$. Clearly, it is sufficient to show that every element of the form $\theta f_{i}$,
where $\theta\in\Theta(r_{1},\dots, r_{p})$ ($1\leq i\leq m$) and $\theta f_{i}\notin\pi(U_{r_{1},\dots, r_{p}})$, is a linear combination of elements of $\pi(U_{r_{1}\dots r_{p}})$ with coefficients in $K$. Indeed, since
$\theta f_{i}\notin\pi(U_{r_{1},\dots, r_{p}})$, $\theta e_{i}\notin U_{r_{1},\dots, r_{p}}$, hence $\theta e_{i}=\theta'u_{g_{j}}^{(1)}$ for some $\theta'\in \Theta$, $1\leq j\leq d$, such that $\ord_{k}(\theta'u_{g_{j}}^{(k)})\leq r_{k}$ for $k=2,\dots, p$. Let us consider the element $g_{j}=a_{j}u^{(1)}_{g_{j}}+ h_{j}$\, ($a_{j}\in K, a_{j}\ne 0$), where $h_{j}$ is a linear combination with coefficients in $K$ of terms that are less than $u_{g_{j}}^{(1)}$ with respect to the order $<_1$. Since $g_{j}\in N=\Ker\,\pi$, $\pi(g_{j})=a_{j}\pi(u_{g_j}^{(1)})+ \pi(h_{j}) = 0$. It follows that
$\pi(\theta'g_{j})=a_{j}\pi(\theta'u_{g_{j}}^{(1)}) + \theta h_{j}) = a_{j}\pi(\theta e_{i}) + \pi(h') = 0$ where $h'$ is a linear combination with coefficients in $K$ of terms $v\in\Theta(r_{1},\dots, r_{p})e$ such that
$v <_{1}\theta e_{i}$. Now we can apply the induction on the well-ordered (with respect to $<_{1}$) set $\Theta e$ to obtain that every element $\theta f_{i}$, such that $\theta\in\Theta(r_{1},\dots, r_{p})$ and $1\leq i\leq m$,
is a linear combination with coefficients in $K$ of element of $\pi(U_{r_{1},\dots, r_{p}})$.
\end{proof}

\begin{thm}
With the notation of Theorem 4.1, there exists a numerical polynomial $\phi_{M}(t_{1}, \dots, t_{p})$ with the following properties.

\smallskip

{\em (i)} $\phi_{M}(r_{1}, \dots, r_{p}) = \dim_{K} M_{r_{1}, \dots, r_{p}}$ for all sufficiently large $(r_{1},\dots, r_{p})\in\mathbb{Z}^{p}$;

\smallskip

{\em (ii)} \, $n_{i}\leq\deg_{t_{i}}\phi_{M}\leq 2n_{i}$ ($1\leq i\leq p$), so that $n\leq deg\,\phi_{M}\leq 2n$ and the polynomial $\phi_{M}(t_{1},\dots, t_{p})$ can be represented as
\begin{equation}
\phi_{M}(t_{1},\dots, t_{p}) = \D\sum_{i_{1}=0}^{2n_{1}}\dots \D\sum_{i_{p}=0}^{2n_{p}}a_{i_{1}\dots i_{p}} {t_{1}+i_{1}\choose i_{1}}\dots {t_{p}+i_{p}\choose i_{p}}
\end{equation}
where $a_{i_{1}\dots i_{p}}\in\mathbb{Z}$ for all $i_{1},\dots, i_{p}$.

\smallskip

{\em (iii)} \, The $A_{n}(K)$-module $M$ is holonomic if and only if $ \deg\phi_{M} = n$.
\end{thm}

\begin{proof}
By Theorem 4.1, the set $\pi(U_{r_{1}\dots r_{p}})$ (we use the above notation) is a basis of the $K$-vector space $M_{r_{1}, \dots, r_{p}}$ ($(r_{1}, \dots, r_{p})\in\mathbb{Z}^{p}$), so
$$\dim_{K} M_{r_{1}, \dots, r_{p}} = \Card U_{r_{1}\dots r_{p}} = \Card V_{r_{1}\dots r_{p}} + \Card V'_{r_{1}\dots r_{p}}.$$
By Theorem 2.5 (applied to the case of partition (3.1) of the set $X\bigcup\Delta$ into $p$ disjoint $2n_{i}$-element subsets $X_{i}\bigcup\Delta_{i}$ ($1\leq i\leq p$), there exists a numerical polynomial $\omega(t_{1},\dots, t_{p})$ in $p$ variables such that $\omega (r_{1},\dots, r_{p}) = \Card\,V_{r_{1},\dots, r_{p}}$
for all sufficiently large  $(r_{1},\dots, r_{p})\in\mathbb{Z}^{p} $ and $\deg_{t_{i}}\omega\leq 2n_{i}$ for $i=1,\dots, p$.

In order to express $\Card V'_{r_{1},\dots, r_{p}}$ in terms of $r_{1},\dots, r_{p}$, let us set $b_{ij} = \ord_{i}u_{g_j}^{(1)}$
and $c_{ij} = \ord_{i}u_{g_j}^{(i)}$ for $i=1,\dots, p;\, j= 1,\dots, d$. Then $b_{1j} = c_{1j}$ and $b_{ij}\leq c_{ij}$ for $i=1,\dots, p;\, j=1,\dots, d$. For any $s=1,\dots, p$ and for any $k_{1}, \dots, k_{s}\in\mathbb{N}$ such that $2\leq k_{1} < \dots < k_{s}\leq p$, let
$$V_{j; k_{1},\dots, k_{s}}(r_{1},\dots, r_{p}) = \{\theta u_{g_{j}}^{(1)}\,|\,\ord_{i}\theta \leq r_{i}-b_{ij}\,\,\, \text{for}\,\,\,   i=1,\dots, p\,\,\, \text{and}\hspace{2in}$$
$$\ord_{l}\theta > r_{l}-c_{lj}\,\,\, \text{if and only if}\,\,\,  l \,\,\, \text{is equal to one of the numbers}\,\,\,  k_{1},\dots ,k_{s}\}.\hspace{2in}$$
By Theorem 2.5(iii), $\Card\Theta(r_{1},\dots, r_{p}) = \D\prod_{i=1}^{p}{r_{i}+2n_{i}\choose 2n_{i}}$ for $r_{1},\dots, r_{p}\in\mathbb{N}$. It follows that
$\Card V_{j; k_{1},\dots, k_{s}}(r_{1},\dots, r_{p}) = \phi_{j; k_{1},\dots, k_{s}}(r_{1},\dots, r_{p})$ where
$$\phi_{j; k_{1},\dots, k_{s}}(t_{1},\dots, t_{p}) = {{t_{1}+2n_{1}-c_{1j}}\choose{2n_{1}}}\dots {{t_{k_{1} -1}+2n_{k_{1}-1}-c_{k_{1} -1,j}}\choose{2n_{k_{1} -1}}}$$
$$\biggl[{{t_{k_{1}}+2n_{k_{1}}-b_{k_{1},j}}\choose{2n_{k_{1}}}} - {{t_{k_{1}}+2n_{k_{1}}-c_{k_{1},j}}\choose{2n_{k_{1}}}}\biggr]{{t_{k_{1}+1}+2n_{k_{1}+1}-c_{k_{1}+1,j}}\choose{2n_{k_{1}+1}}}$$
$$\dots {{t_{k_{s} -1}+2n_{k_{s} -1}-c_{k_{s}-1,j}}\choose{2n_{k_{s} -1}}}\biggl[{{t_{k_{s}}+2n_{k_{s}}-b_{k_{s},j}}\choose{2n_{k_{s}}}} - {{t_{k_{s}}+2n_{k_{s}}-c_{k_{s},j}}\choose{2n_{k_{s}}}}\biggr]$$
\begin{equation}
\dots {{t_{p}+2n_{p}-c_{pj}}\choose{2n_{p}}}.\hspace{3.2in}
\end{equation}
Applying the combinatorial principle of inclusion and exclusion we obtain that for any $r_{1}, \dots, r_{p}$,  $\Card V'_{r_{1},\dots, r_{p}}$ is an alternating sum of the numbers of the form
$$\Card V_{j; k_{1},\dots, k_{s}}(r_{1},\dots, r_{p}) = \phi_{j; k_{1},\dots, k_{s}}(r_{1},\dots, r_{p})$$ where $1\leq j\leq d$, $1\leq s\leq p$, $2\leq k_{1} < \dots < k_{s}\leq p$. Since
the degree of any polynomial $\phi_{j; k_{1},\dots, k_{s}}(t_{1},\dots, t_{p})$ with respect to $t_{i}$ ($1\leq i\leq p$) does not exceed $2n_{i}$, we obtain that there exists a numerical polynomial $\psi(t_{1},\dots, t_{p})$ such that $\psi(r_{1},\dots, r_{p}) = \Card V'_{r_{1},\dots, r_{p}}$ for all sufficiently large $(r_{1},\dots, r_{p})\in\mathbb{N}^{p}$  and $\deg_{t_{i}}\psi \leq 2n_{i}$ for $i=1,\dots, p$. Thus, the numerical polynomial
$$\phi_{M}(t_{1}, \dots, t_{p}) = \omega(t_{1},\dots, t_{p}) + \psi(t_{1},\dots, t_{p})$$ satisfies condition (i) of the theorem and $\deg_{t_{i}}\phi_{M}\leq  2n_{i}$ ($1\leq i\leq p$), so it remains to show that $\deg_{t_{i}}\phi_{M}\geq n_{i}$ for $i=1,\dots, p$.

We will prove the last inequalities by using the idea of the A. Joseph's proof of the Bernstein inequality for (univariate) Bernstein polynomial (see \cite[Chapter 9, Section 4]{Co}). Let $(r_{1},\dots, r_{p})\in\mathbb{N}^{p}$ and let a mapping
$$\Phi:W_{r_{1},\dots, r_{p}}\rightarrow \Hom_{K}(M_{r_{1},\dots, r_{p}}, M_{2r_{1},\dots, 2r_{p}})$$ be defined by $\Phi(D) = \Phi_{D}$ where $\Phi_{D}(z) = Dz$ for any $z\in M_{r_{1},\dots, r_{p}}$.
We are going to show that $\Phi$ is injective by induction on $r_{1}+\dots + r_{p}$.
If this number is $0$, that is, $(r_{1},\dots, r_{p})=(0,\dots, 0)$, then $D\in K$ and it is clear that the equality $\Phi_{D}=0$ implies $D=0$. Let $r_{1}+\dots + r_{p} > 0$ and let $D$ be an element of $W_{r_{1},\dots, r_{p}}\setminus K$ such that $\Phi_{D}M_{r_{1},\dots, r_{p}}=0$. Then $D$ has a canonical (unique) representation as a finite sum $\sum_{\alpha, \beta} c_{\alpha \beta}x^{\alpha}\partial^{\beta}$ ($c_{\alpha \beta}\in K$) and there exists a nonzero coefficient $c_{\alpha\beta}$ such that $|\alpha| + |\beta| > 0$, that is, either $|\alpha| > 0$ or $|\beta| > 0$. If $|\alpha| > 0$ and $\alpha_{k} > 0$ ($1\leq k\leq n$), where $x_{k}\in X_{i}$ ($1\leq i\leq p$, see (3.1)), then $\alpha_{k}c_{\alpha\beta}x_{1}^{\alpha_{1}}\dots x_{k-1}^{\alpha_{k-1}}x_{k}^{\alpha_{k}-1}x_{k+1}^{\alpha_{k+1}}\dots x_{n}^{\alpha_{n}}\partial^{\beta}$ is a summand in the canonical representation of $[D, \partial_{k}] = D\partial_{k} - \partial_{k}D$. Therefore, $[D, \partial_{k}]\neq 0$. At the same time,
$[D, \partial_{k}]\in W_{r_{1},\dots, r_{i-1}, r_{i}-1, r_{i+1},\dots, r_{p}}$ and
$$[D, \partial_{k}]M_{r_{1},\dots, r_{i-1}, r_{i}-1, r_{i+1},\dots, r_{p}} = D\partial_{k}M_{r_{1},\dots, r_{i-1}, r_{i}-1, r_{i+1},\dots, r_{p}} - \partial_{k}DM_{r_{1},\dots, r_{i-1}, r_{i}-1, r_{i+1},\dots, r_{p}}$$
$ = 0,$ because $\partial_{k}M_{r_{1},\dots, r_{i-1}, r_{i}-1, r_{i+1},\dots, r_{p}}\subseteq M_{r_{1},\dots, r_{p}}$. Since $[D, \partial_{k}]\neq 0$, we obtain a contradiction with the induction hypothesis. In a similar way, if $|\beta| > 0$ (so that $\beta_{k}> 0$ for some $k\in\{1,\dots, n\}$), we obtain that if $\partial_{k}\in\Delta_{i}$ ($1\leq i\leq p$) and $\Phi_{D}M_{r_{1},\dots, r_{p}}=0$, then $[D, x_{k}]\neq 0$ and $[D, \partial_{k}]M_{r_{1},\dots, r_{i-1}, r_{i}-1, r_{i+1},\dots, r_{p}} = 0$, contrary to the induction hypothesis. Thus, the mapping $\Phi$ is injective. It follows that
$$\dim_{K}W_{r_{1},\dots, r_{p}}\leq \dim_{K}\left(\Hom_{K}(M_{r_{1},\dots, r_{p}}, M_{2r_{1},\dots, 2r_{p}})\right)$$
\begin{equation}
= \phi_{M}(r_{1},\dots, r_{p})\phi_{M}(2r_{1},\dots, 2r_{p})\hspace{1.4in}
\end{equation}
for all sufficiently large $(r_{1},\dots, r_{p})\in\mathbb{Z}^{p}$. Since $\dim_{K}W_{r_{1},\dots, r_{p}} = \D\prod_{i=1}^{p}{r_{i}+2n_{i}\choose 2n_{i}}$, we obtain (by fixing all $r_{j}$ with $j\neq i$ and considering the limit as $i\rightarrow\infty$ in (4.3)) that for every $i=1,\dots, p$, one has $2n_{i}\leq 2\deg_{t_{i}}\phi_{M}$, so $\deg_{t_{i}}\phi_{M}\geq n_{i}$. Also, $2\deg\phi_{M}\geq 2(n_{1}+\dots + n_{p}) = 2n$, so $\deg\phi_{M}\geq n$.

In order to proof the last statement of the theorem, note that for all sufficiently large $r\in\mathbb{N}$, $\phi_{M}(r,\dots, r)\leq \psi_{M}(pr)$ ($\psi_{M}$ is the Bernstein polynomial of $M$, see Theorem 2.1) and
$\psi_{M}(r)\leq \phi_{M}(r, \dots, r)$. Therefore, if the module $M$ is holomorphic, then $n\leq\deg\phi_{M} = \deg\phi_{M}(t,\dots, t)\leq \deg\psi_{M}(pt) = n$, so $\deg \phi_{M} = n$. Conversely, if
$\deg\phi_{M} = n$, then $$n\leq \deg\psi_{M}(t)\leq\deg\phi_{M}(t, \dots, t) = n,$$ so the module $M$ is holonomic.
\end{proof}
\begin{defn}
The polynomial $\phi_{M}(t_{1}, \dots, t_{p})$, whose existence is established by Theorem 4.2, is called a {\em dimension polynomial} of the $D$-module $M$ associated with the given system of generators $\{f_{1},\dots, f_{m}\}$.
\end{defn}

Generally speaking, different finite systems of generators of $M$ over $A_{n}(K)$ produce different dimension polynomials, however every dimension polynomial carries certain integers that do not depend on generators it is associated with. These integers, that characterize the $D$-module $M$, are called {\em invariants} of a dimension polynomial.  We describe some of such invariants in the next theorem.

\smallskip

For any permutation $(j_{1},\dots, j_{p})$ of the set $\{1,\dots, p\}$, let $\leq_{j_{1},\dots, j_{p}}$ denote a lexicographic order on $\mathbb{N}^{p}$ defined as follows:
$(r_{1},\dots, r_{p})\leq_{j_{1},\dots, j_{p}} (s_{1},\dots, s_{p})$ if and only if either $r_{j_{1}} < s_{j_{1}}$ or there exists $k\in\mathbb{N}$, $1\leq k\leq p-1$, such that $r_{j_{\nu}} = s_{j_{\nu}}$ for $\nu
= 1,\dots, k$ and $r_{j_{k+1}} < s_{j_{k+1}}$. Furthermore, for any set $S\subseteq\mathbb{N}^{p}$, let $S'$ denote the set $\{a\in S\,|\, a$ is a maximal element of $S$ with respect to one of the $p!$ lexicographic orders
$\leq_{j_{1},\dots, j_{p}}\}$.

\smallskip

For example, if $S = \{(1, 1, 2), (3, 1, 1), (2, 3, 0), (3, 0, 2), (1, 4, 0), (2, 3, 1),\\ (0, 4, 1), (0, 3, 3), (1, 0, 3)\}\subseteq\mathbb{N}^{3}$, then
$S' = \{(3, 1, 1), (3, 0, 2), (1, 4, 0), (0, 4, 1), \\(0, 3, 3), (1, 0, 3)\}$.

\begin{thm}
With the above notation, let $M$ be a finitely generated $D$-module and
$$\phi_{M}(t_{1}, \dots, t_{p}) = \D\sum_{i_{1}=0}^{n_{1}}\dots \D\sum_{i_{p}=0}^{n_{p}}a_{i_{1}\dots i_{p}} {t_{1}+i_{1}\choose i_{1}}\dots {t_{p}+i_{p}\choose i_{p}}$$ the dimension polynomial
associated with some finite system of generators $\{f_{1}, \dots, f_{m}\}$ of $M$. (We write $\phi_{M}$ in the form {\em (2.1)} with integer coefficients $a_{i_{1}\dots i_{p}}$.) Let $S(\phi_{M}) = \{(i_{1},\dots, i_{p})\in\mathbb{N}^{p}\,|\, 0\leq i_{k}\leq n_{k}$ $(k=1,\dots, p)$ and $a_{i_{1}\dots i_{p}}\neq 0\}$. Then $d = \deg\,\phi_{M}$, $a_{n_{1}\dots n_{p}}$, elements $(k_{1},\dots, k_{p})\in S(\phi_{M})'$, the corresponding coefficients $a_{k_{1}\dots k_{p}}$, and the coefficients of the terms of total degree $d$ do not depend on the finite system of generators of the $A_{n}(K)$-module $M$ this polynomial is associated with.
\end{thm}

\begin{proof}
Let $\{h_{1}, \dots, h_{q}\}$ be another finite system of generators of the $A_{n}(K)$-module $M$ and let $M_{r_{1},\dots, r_{p}} = \D\sum_{i=1}^{m}W_{r_{1},\dots, r_{p}}f_{i}$ and
$M'_{r_{1},\dots, r_{p}} = \D\sum_{i=1}^{q}W_{r_{1},\dots, r_{p}}h_{i}$ ($(r_{1},\dots, r_{p})\in\mathbb{Z}^{p}$).  Let
$$\phi^{\ast}_{M}(t_{1}, \dots, t_{p}) = \D\sum_{i_{1}=0}^{n_{1}}\dots \D\sum_{i_{p}=0}^{n_{p}}b_{i_{1}\dots i_{p}} {t_{1}+i_{1}\choose i_{1}}\dots {t_{p}+i_{p}\choose i_{p}}$$
be the dimension polynomial associated with the system of generators $\{h_{1}, \dots, h_{q}\}$ ($b_{i_{1}\dots i_{p}}\in\mathbb{Z}$). Then there exists $(s_{1},\dots, s_{p})\in\mathbb{N}^{p}$ such that
$f_{i}\in \D\sum_{j=1}^{q}W_{s_{1},\dots, s_{p}}h_{j}$ ($1\leq i\leq m$) and $h_{j}\in \D\sum_{i=1}^{m}W_{s_{1},\dots, s_{p}}f_{i}$ ($1\leq j\leq q$). It follows that for all sufficiently large
$(r_{1},\dots, r_{p})\in\mathbb{Z}^{p}$, one has $M_{r_{1},\dots, r_{p}}\subseteq M'_{r_{1}+s_{1},\dots, r_{p}+s_{p}}$ and $M'_{r_{1},\dots, r_{p}}\subseteq M_{r_{1}+s_{1},\dots, r_{p}+s_{p}}$, that is,
\begin{equation}
\phi_{M}(r_{1},\dots, r_{p})\leq \phi_{M}^{\ast}(r_{1}+s_{1},\dots, r_{p}+s_{p})
\end{equation} and
\begin{equation}
\phi^{\ast}_{M}(r_{1},\dots, r_{p})\leq \phi_{M}(r_{1}+s_{1},\dots, r_{p}+s_{p}).
\end{equation}
If we set $r_{i}= r$ for $i=1,\dots, p$ and let $r\rightarrow\infty$ we obtain that $\phi_{M}(t_{1},\dots, t_{p})$ and $\phi_{M}^{\ast}(t_{1},\dots, t_{p})$ have the same degree $d$ and the same
coefficient of the monomial $t_{1}^{n_{1}}\dots t_{p}^{n_{p}}$. If $(k_{1},\dots, k_{p})$ is the maximal element of $S(\phi_{M})'$ with respect to the lexicographic order $\leq_{j_{1},\dots, j_{p}}$, then we set $r_{j_{p}} = r$, $r_{j_{p-1}} = 2^{r_{j_{p}}}= 2^{r}, \dots, r_{j_{1}} = 2^{r_{j_{2}}}$ and let $r\rightarrow\infty$. Then the term $r_{1}^{k_{1}}\dots r_{p}^{k_{p}}$ becomes the greatest one in the canonical representations of $\phi_{M}$ and $\phi_{M}^{\ast}$ with $t_{i}=r_{i}$ ($1\leq i\leq p$). Now, the equalities (4.4) and (4.5) show that $(k_{1},\dots, k_{p})$ is also a maximal element of $S(\phi^{\ast}_{M})'$ with respect to $\leq_{j_{1},\dots, j_{p}}$
and the coefficients of $t_{1}^{k_{1}}\dots t_{p}^{k_{p}}$ in the polynomials $\phi_{M}(t_{1},\dots, t_{p})$ and $\phi^{\ast}_{M}(t_{1},\dots, t_{p})$ are equal.

Finally, let us order the terms of the total degree $d$ in $\phi_{M}$ and $\phi^{\ast}_{M}$ using the lexicographic order $\leq_{p,p-1,\dots, 1}$ on $\mathbb{N}^{p}$, set
$$w_{1} = r, w_{2} = 2^{w_{1}} =2^{r}, \dots, w_{p} = 2^{w_{p-1}}, T = 2^{w_{p}}, t_{i} = w_{i}T\, (1\leq i\leq p)$$ and let $r\rightarrow\infty$. Then the inequalities (4.4) and (4.5) imply that the polynomials $\phi_{M}(t_{1},\dots, t_{p})$ and $\phi^{\ast}_{M}(t_{1},\dots, t_{p})$ have the same maximal multi-degree with respect to $\leq_{p,p-1,\dots, 1}$ among multi-degrees of monomials of total degree $d$, and also coefficients of the corresponding monomials in $\phi_{M}$ and $\phi^{\ast}_{M}$ are equal. In the same way, we obtain similar equalities of terms of total degree $d$ that have maximal multi-degree with respect to each of the $p!$ lexicographic orders $\leq_{j_{1},\dots, j_{p}}$ on $\mathbb{N}^{p}$. This completes the proof. \,$\Box$
\end{proof}

\smallskip

The following example uses the last theorem to illustrate the fact that a multivariate dimension polynomial introduced in Theorem 4.2 carries more invariants of a finitely generated $D$-module than the classical (univariate)
Bernstein polynomial.

\begin{ex}
With the above notation, let $n = 2$, $X_{1} = \{x_{1}\}$, $X_{2} = \{x_{2}\}$, $\Delta_{1}=\{\partial_{1}\}$, and $\Delta_{2}=\{\partial_{2}\}$. Let a left $A_{2}(K)$-module $M$ be generated by one
element $f$ that satisfies the defining equation
\begin{equation}
x_{1}^{\alpha}\partial_{2}^{\beta}f + x_{2}^{\gamma}\partial_{1}^{\alpha}f = 0
\end{equation}
where $\alpha, \beta, \gamma\in\mathbb{N}$ and $\gamma > \beta$. In other words, $M$ is a factor module of a free $A_{2}(K)$-module $E = A_{2}(K)e$ with a free generator $e$ by its $A_{2}(K)$-submodule
$N = A_{2}(K)(x_{1}^{\alpha}\partial_{2}^{\beta} + x_{2 }^{\gamma}\partial_{1}^{\alpha})e$. \, It is easy to see that the set consisting of a single element $g = (x_{1}^{\alpha}\partial_{2}^{\beta} + x_{2}^{\gamma}\partial_{1}^{\alpha})e$ is a Gr\"obner basis of $N$ with respects to the orders $<_{1}, <_{2}$. In this case, the proof of Theorem 4.1 (with the use of formula (2.2)) shows that the dimension polynomial of $M$ associated with the natural $2$-dimensional bifiltration $\{M_{rs} = W_{rs}z\,|\,r, s\in\mathbb{Z}\}$ is as follows:
$$\phi_{M}(t_{1}, t_{2}) = \left[{t_{1}+2\choose 2}{t_{2}+2\choose 2} - {t_{1}+2-\alpha\choose 2}{t_{2} + 2-\beta\choose 2}\right] + \hspace{3in}$$
$${t_{1}+2-\alpha\choose 2}\left[{t_{2} + 2-\beta\choose 2} - {t_{2} + 2-\gamma\choose 2}\right]  = {\frac{1}{2}}\gamma t_{1}^{2}t_{2} + {\frac{1}{2}}\alpha t_{1}t_{2}^{2} + o(t_{1}, t_{2})$$
where $o(t_{1}, t_{2})$ is a polynomial in $\mathbb{Q}[t_{1}, t_{2}]$ of total degree at most $2$. (With the notation of the proof of Theorem 4.1,  the polynomial in the first brackets gives $\Card\,V_{r_{1},\dots, r_{p}}$ while the polynomial $\D{t_{1}+2-\alpha\choose 2}\left[{t_{2} + 2-\beta\choose 2} - \D{t_{2} + 2-\gamma\choose 2}\right]$ gives $\Card\,V'_{r_{1},\dots, r_{p}}$ for all sufficiently large $(r_{1},\dots, r_{p})\in\mathbb{Z}^{p}$.)

The univariate Bernstein polynomial associated with the (one-dimensional) filtration $\{W_{r}e\,|\,r\in\mathbb{Z}\}$ is
$$\psi_{M}(t) = {{t+4\choose 4}} - {{t+4-(\alpha+\gamma)\choose 4}} = {\frac{\alpha+\gamma}{3!}}t^{3}+ o(t^{3}).$$
(It follows from the fact that $\{g\}$ is obviously a Gr\"obner basis of $N$ in the usual sense, see \cite[Section 1.1]{SST}, and the term of $g$ of the maximal total degree is $x_{2}^{\gamma}\partial_{1}^{\alpha}e$.)

Comparing this polynomials $\psi_{M}(t)$ and $\phi_{M}(t_{1}, t_{2})$, we see that the Bernstein polynomial $\psi_{M}(t)$ carries two invariants of the module $M$, its degree 3 and the multiplicity $\alpha+\gamma$. At the same time, the dimension polynomial $\phi_{M}(t_{1}, t_{2})$, according to Theorem 4.4, carries three such invariants, its total degree 4, $\alpha$, and $\gamma$.  Therefore, the polynomial $\phi_{M}(t_{1}, t_{2})$ determines
two parameters $\alpha$ and $\gamma$ of the defining equation (4.6) while $\psi_{M}(t)$  gives just the sum of these parameters.
\end{ex}

\section{Acknowledges}

This research was supported by the NSF grant CCF-2139462.


\begin{thebibliography}{1}

\bibitem{B1}
Bernstein, I. N.
\textit{Modules over the ring of differential operators. A study of the fundamental solutions of equations with constant coefficients}. Funct. Anal. and its Appl., 5 (1971) 89--101.

\bibitem{B2}
Bernstein, I. N.
\textit{The analytic continuation of generalized functions with respect to a parameter}. Funct. Anal. and its Appl., 6 (1972) 273 - 285.

\bibitem{Bj}
Bj\"ork, J.-E.
\textit{Rings of Differential Operators}. North Holland Publishing Company, Amsterdam, New York,  1979.

\bibitem{BM}
Brian, J., Maisonobe, Ph.
\textit{Id\'{e}aux de germes d'op\'{e}rateurs diff\'{e}rentiels \'{a} une variable}. Enseign. Math. (2) 30 (1984), no. 1--2, 7--38.

\bibitem{CJ}
Castro-Jim\'{e}nez, F. J.
\textit{Th\'{e}or\`{e}me de division pour les op\'{e}rateurs diff\'{e}rentiels et calcul des multiplicit\'{e}s}. PhD Thesis, Univ. Paris VII, October 1984.

\bibitem{Co}
Coutinho, S. C.
\textit{A Primer of Algebraic D-modules}. Cambridge Univ. Press, 1995.

\bibitem{DL}
D\"onch, C.; Levin, A. B.
\textit{Bivariate Dimension Polynomials and New Invariants of Finitely Generated D-Modules}. Int. J. Algebra Comput., 23, no. 7 (2013) , 1625--1651.

\bibitem{K}
Kolchin, E. R.
\textit{Differential Algebra and Algebraic Groups}. Academic Press, 1973.

\bibitem{KLMP}
Kondrateva, M. V., Levin, A. B.,Mikhalev, A. V., Pankratev, E. V.
\textit{Differential and Difference Dimension Polynomials}. Kluwer Acad. Publ., 1999.

\bibitem{Levin1}
Levin, A. B.
\textit{Gr\"obner bases with respect to several orderings and multivariable dimension polynomials}.  J. Symb. Comput., 42 (2007), no. 5, 561--578.

\bibitem{Levin2}
Levin, A. B.
\textit{Characteristic Polynomials of Finitely Generated Modules over Weyl Algebras}. Bull. Aust. Math. Soc., 61 (2000), no. 3, 387--403.

\bibitem{N}
Noro, M.
\textit{Gr\"obner basis calculation on Weyl algebra}. In: Risa/Asir and their applications. S\={u}rikaisekikenky\={u}sho K\={o}ky\={u}roku, (1199), 2001,43--50. Theory and
application in computer algebra (Japanese) (Kyoto, 2000).

\bibitem{O1}
Oaku, T.
\textit{An algorithm of computing b-functions}. Duke Math. J., 87(1997), no. 1, 115--132.

\bibitem{O2}
Oaku, T.
\textit{Some algorithmic aspects of the D-module theory}. In: New trends in microlocal analysis (Tokyo, 1995), 205--223, Springer, Tokyo, 1997.

\bibitem{OT}
Oaku, T., Takayama, N.
\textit{Algorithms for D-modules—restriction, tensor product, localization, and local cohomology groups}. J. Pure Appl. Algebra, 156 (2001), no. 2--3, 267--308.

\bibitem{SST}
M. Saito, M., Sturmfels, B., Takayama, N.
\textit{Gr\"obner deformations of hypergeometric differential equations}. Algorithms and Computation in Mathematics, 6.
Springer-Verlag, Berlin, (2000).

\end{thebibliography}
\end{document}